\documentclass[11pt]{article}
\usepackage{amsmath,amsthm,amssymb}
\usepackage[applemac]{inputenc}
\usepackage{amsfonts}
\theoremstyle{theorem}
\newtheorem{theorem}{Theorem}
\newtheorem{lemma}[theorem]{Lemma}
\theoremstyle{remark}

\numberwithin{equation}{section}

\newcommand{\ve}{\varepsilon}

\newcommand{\R}{\mathbb{R}}

\newcommand\huno{{\mathcal H}^1}

\DeclareMathOperator{\osc}{osc}

\begin{document}
\author{Filippo Santambrogio\thanks{\scriptsize\ CEREMADE, UMR CNRS 7534, Universit\'e Paris-Dauphine, Pl. de Lattre de Tassigny, 75775 Paris Cedex 16, FRANCE
\texttt{filippo@ceremade.dauphine.fr}}, Vincenzo Vespri\thanks{\scriptsize\ 
Dipartimento di Matematica Ulisse Dini,
Università degli studi di Firenze, 
Viale Morgagni, 67/a 50134 Firenze, ITALY 
\texttt{vespri@math.unifi.it}}}
\title{Continuity in two dimensions for a very degenerate elliptic equation}
\maketitle

{\bf Abstract:} An elliptic equation $\nabla\cdot (F(\nabla u))=f$ whose ellipticity strongly degenerates for small values of $\nabla u$ (say, $F=0$ on $B(0,1)$) is considered. The aim is to prove regularity for $F(\nabla u)$. The paper proves a continuity result in dimension two and presents some applications.

\section{Introduction}

We consider in this paper an elliptic equation of the form 
\begin{equation}\label{1.1}
\nabla\cdot (F(\nabla u))=f,
\end{equation}
where $F$ is the gradient of a convex function $H:\R^d\to\R$. This is one of the most classical non-linear elliptic equations, which arises as an optimality condition for the minimization of the functional $\int H(\nabla u)+\int fu$. When $H(z)=|z|^2$ or $H(z)=|z|^p$ we get the usual Laplace and $p-$Laplace equations, respectively, which have been investigated intensively in the literature and have provided a lot of regularity results on the solution $u$ (which is a priori supposed to belong to $H^1$ or $W^{1,p}$ only) according to the regularity of $f$. Most of the results have been extended to the case of variable coefficients (see the book by Gilbarg and Trudinger \cite{GilTru} for a compendium of the theory of elliptic regularity and the original paper by De Giorgi, \cite{DeG57}, where he proved a key H\"older regularity result) or of different functions $H$, which share anyway some properties of the square or of the $p-$th power. This latter case is much more difficult than that of the square (which gives a linear equation), mainly because of the degeneracy of $D^2H$ near $z=0$. Actually, lower bounds on $D^2H$ are often useful to give estimates on the norms of the solution in terms of the norms of $f$ and if $D^2H$ is allowed to tend to zero for small values for $\nabla u$ (which is the case fro $p>2$), extra difficulties arise. Yet, this is - roughly speaking - compensated by the fact that the smallness of the gradient already provides some regularity estimates and some results are still possible.

This is what is usually done in elliptic regularity but the equation we want to look here is even worse than the $p-$Laplace equation. Our main example is given by the function 
\begin{equation}\label{Hq}
H_{(q)}(z)= \frac 1q \left(|z|-1\right)_+^q,\qquad q>1
\end{equation}
which vanishes, together with its Hessian, on the whole ball $B_1$. This means that all the values of $\nabla u$ which are smaller than $1$ are a source of problems. Obviously, the regularity result that one may expect to prove must concern $F(\nabla u)$ instead of $u$ itself: it is evident that the homogeneous equation (the same, with $f=0$) would be solved by any $1-$Lipschitz function and that nothing more may be said on $u$. On the contrary, $F(\nabla u)$ is reasonably more regular, since either it vanishes or we are in a zone where $|\nabla u|>1$ and the equation is more elliptic. But gluing the two zones, which are not open sets with nice boundaries, is not trivial.

Up to our knowledge, the first need for a regularity result on $F(\nabla u)$ in the case of a degenerate function $F$ of this kind has been encountered in \cite{BraCarSan} for a variational problem linked to traffic congestion. In this problem, i.e. 
\begin{equation}\label{PHq}
\min \left\{ \int |\sigma|+\frac 1q |\sigma|^q\;:\; \sigma\in L^q(\Omega),\;\nabla\cdot\sigma = f,\;\sigma\cdot n =0\right\}
\end{equation}
the optimal $\bar{\sigma}$ equals $F(\nabla u)$ where $u$ is a solution of \eqref{1.1} with $F=\nabla H_{(q)}$, Neumann boundary conditions and $\frac 1p + \frac 1q =1$. For the sake of the applications of \cite{BraCarSan} boundedness and Sobolev regularity were all that was needed on $\bar{\sigma}$, and they are proven for functions $H$ of the form given in \eqref{Hq}. In particular it is proven that $\bar{\sigma}=F(\nabla u)$ belongs to $H^1\cap L^\infty$ in any dimension $d$. Continuity is not addressed, even if it would be interesting to obtain. We will explain in Section 6 the possible applications of a continuity result, both in the framework of traffic congestion and of the variational problem \eqref{PHq} in general.

Here we consider the more general case of a function $H$ whose Hessian is bounded from below outside any ball $B_{1+\delta}$ and we prove a continuity result which is valid in two dimensions, provided that we already know that $F(\nabla u)$ belongs to $H^1\cap L^\infty$. This is exactly the case for $H=H_{(q)}$. Yet, the assumptions on $f$ are more standard in this framework ($f\in L^{2+\ve}$), while the Sobolev results of \cite{BraCarSan} asked for Sobolev regularity of $f$ itself, which is not that natural in elliptic regularity. 

The proof of the present paper is two-dimensional because it follows a method developed by DiBenedetto and Vespri for different equations, which is based on the following idea: using the equation, one can prove that if the oscillation is not significantly reduced when passing from $B_R$ to $B_{\ve_0R}$, then the contribution of the crown $B_R\setminus B_{\ve_0R}$ to the Dirichlet energy (which is finite because of the $H^1$ assumption) is at least a certain value. This value, in the case of $d=2$, scales in a way so that it does not depend on $R$. This is used so as to prove a decay for the oscillation and get in the end a logarithmic modulus of continuity.

In our case we first prove results for functions of the type $(\partial u/\partial x_i - (1+\delta))_+$ and then deduce some results at the limit as $\delta\to 0$. This gives continuity of all the expression $g(\nabla u)$ for all the functions $g$ vanishing on $B_1$. Typically, this includes $F(\nabla u)$ itself.

Notice that all the results we present are local and that we did not try to improve them up to the boundary.

The strategy followed by the paper and its plan will be detailed in Section 2 below.

\section{The equation, the assumptions and the strategy}

Let $H:\R^d\to \R$ be a convex function, satisfying the key assumption that 
\begin{equation}\label{ass on H}
\mbox{ for all $\delta>0$ there exists } c_\delta>0\;:\; D^2H(z)\geq c_\delta I_d \mbox{ for all $z$ such that }|z|\geq 1+\delta,
\end{equation}
where $I_d$ is the identity matrix. The function $H$ may be for simplicity supposed $C^2$, (or $C^{1,1}_{loc}$). For notational purposes, as we already did in the introduction, we will denote $\nabla H$ by $F$, so that $F:\R^d\to \R^d$ is a vector function satisfying some monotonicity assumptions.
We will deal with solutions of the following equation:
$$\nabla\cdot (F(\nabla u))=f,$$ 
being $f$ a given datum in a suitable space. We will see that for most of the results we prove $L^{2+\ve}$ will be sufficient, while for some of those that we recall we would need $f$ in a Sobolev space.

Notice that this is the equation which is satisfied by the minimizers of
$$\min \int H(\nabla u) + \int fu$$
(we do not precise boundary conditions since  we are only interested in local results).
The function $H$ is allowed, for instance, to vanish on the whole $B_1$, as it is the case for the typical example $H=H_{(q)}$. Hence, $H$ is in general not strictly convex and no uniqueness is valid for the minimization problem above. Yet, it is easy to see that $F(\nabla u)$ is the same for all the minimizers. This vector field will be the main object of investigation. 
By the way, $F(\nabla u)$ is also the solution of the dual problem
$$\min \int H^*(\sigma) \;:\; \nabla\cdot\sigma=f.$$

\paragraph{Strategy.} For a solution $u$ of our problem, we write down the equation which is satisfied by a partial derivative $v=\partial u / \partial x_1$.

To do so, we are tacitly assuming that $u$ is regular (at least $H^2$), which is not at all obvious with the kind of equation we have. Hence the idea is: approximate the equation by truly elliptic equations (for instance changing the function $H$ into a function $H_\ve$, still satisfying \eqref{ass on H}, but satisfying also $C^-_\ve\leq D^2H_\ve\leq C^+_\ve$, for some non-uniform constant $C^\pm_\ve$, and fixing suitable boundary values: see for instance \cite{BraCarSan}); then write estimates on the solutions $u_\ve$ which will only depend on $c_\delta$ and not on the constants $C^\pm_\ve$. This estimate will pass to the limit as $\ve\to 0$. The functions $u_\ve$ will converge up to subsequences to a solution of the limit equation, which has no uniqueness, but the quantity $F(\nabla u)$ is the same for all the solutions. Hence, estimates on $F(\nabla u)$ will be true for the non-approximated equation as well. 

In the following, we will not precise any more the approximation, but ,with this approximated spirit in mind, let us differentiate the equation and get
$$\nabla\cdot (a(x) \cdot \nabla v)=f',$$ 
being $a(x)=D^2H(\nabla u(x))$ the Hessian matrix of $H$ and $f'$ the partial derivative of $f$. This is a linear elliptic equation with the following property: $a$ is not uniformly elliptic and can also vanish, but on the points where $|v(x)|\geq 1+\delta$ one has a fortiori $|\nabla u(x)|\geq 1+\delta$ and hence $a(x)\geq c_\delta I_d$. 

The main object of our preliminary analysis will be the function $v_\delta=h_{1+\delta}(v)$ (we define $h_{k}(t):=(t-k)_+$). Since $h_{1+\delta}$ is a Lipschitz and convex function, $v_\delta$ is a subsolution of
$$\nabla\cdot (a(x) \cdot \nabla v_\delta)= f'I_{v>1+\delta}.$$ 

This new equation may be assumed to be uniformly elliptic, since the values of the matrix $a$ on the points where $v_\delta=0$ are not important. Yet, $v_\delta$ is only a subsolution and this will not be sufficient to establish all the results we need for regularity. At a given point, we will also need to use the equation which is satisfied by $v$ itself (which is not uniformly elliptic).

In Section 3 we will prove a continuity result for $v_\delta$, in dimension 2 and under some a priori assumption on $v_\delta$ itself. We will need to suppose that $v_\delta$ is a $H^1$ function and that it is bounded. Notice that, due to the fact that we tacitly regularize, when we say that $v_\delta$ or $f$ belong to a certain functional space, we actually mean that the estimates we prove will only depend on the norm in this space, and hence will be inherited by the solutions corresponding to non-regular data, provided we stay in the same space. Also the continuity result is given in the same spirit, in the sense that we will prove a quantified continuity (in our case, the modulus of continuity will be $\omega(R)=C|\ln R|^{-1/2}$ and the constants will depend on $c_\delta$). The techniques follow a strategy by DiBenedetto and Vespri (see \cite{DiBVes}) which amounts at proving some oscillation decay from a ball to a smaller one under some conditions. These conditions are such that violating them implies that the solution accumulates a fixed amount of Dirichlet Energy in the crown between the two balls and the $H^1$ assumption allows to conclude.\\
Section 3 aims at being as general as possible and some of the lemmas will be stated in $\R^d$ instead of $\R^2$.

In Section 4 we will see how to link $v_\delta$ to $F(\nabla u)$. In particular this will allow to translate Sobolev results on $F(\nabla u)$ into corresponding results for $v_\delta$, and also to state that $v_\delta$ as well does not depend on the non-unique solution $u$.

In Section 5 we will conclude wider continuity results on functions of $\nabla u$, as a consequence of the results on $v_\delta$.

Section 6 will present some interesting applications, in cases when we do actually know that $F(\nabla u)$ is Sobolev and bounded. 
This is for instance the case for $H=H_{(q)}$ (see \cite{BraCarSan}). The $H^1$ regularity mainly depends on structure assumptions on the function $H$ and may fail to be valid in general, while the $L^\infty$ one seems to be easier, since it only depends on the fact that functions of the kind $(\partial u/\partial x_i - (1+\delta))_+$ are  subsolution of a uniformly elliptic equation.

\section{Continuity for $v_\delta$ in dimension 2}

The section will be opened by some very classical lemmas which have proven to be useful several times in elliptic regularity. For brevity, we omit the proofs and we refer to the book by DiBenedetto \cite{DiBbook}. We also refer to the original work by De Giorgi \cite{DeG57}.

\begin{lemma}\label{algebra}
If $(Y_n)_n$ is a sequence of non-negative numbers satisfying
$$Y_{n+1}\leq cb^nY_n^{1+\beta},\quad Y_1\leq c^{-1/\beta}b^{-(\beta+1)/\beta^2},\quad c,b,\beta>0,$$
then $Y_n\to 0$.
\end{lemma}

\begin{lemma}\label{poincare type}
The following Poincaré-type inequality holds on any domain $\Omega$ in $\R^d$ with a universal constant (depending on the dimension $d$ only): if $u\in H^1_0(\Omega)$, then
$$\int_\Omega u^2 \leq C\left|\{u\neq 0\}\right|^{2/d}\int_\Omega |\nabla u|^2.$$
\end{lemma}

If the previous inequality comes from an idea by De Giorgi, the following is very very standard:
\begin{lemma}\label{poincare non so}
If a shape $\Omega_1$ is fixed, then for any rescaled open set $\Omega_R=R\Omega_1\subset\R^d$ the following Poincaré-type inequality holds with a constant $C$ depending on $\Omega_1$ and on the ratio $\lambda$:
if $u\in H^1(\Omega_R)$ and $|\{u=0\}|\geq \lambda |\Omega_R|$, then
$$\int_{\Omega_R}u^2 \leq CR^2\int_{\Omega_R} |\nabla u|^2.$$
\end{lemma}

\begin{lemma}\label{reduction from above}
Suppose $f\in L^{2+\ve}$, and let $v$ be a subsolution of 
$$\nabla\cdot (a(x) \cdot \nabla v)=f'I_{v>0},$$ 
where the matrix $a$ satisfies $\bar{c}I_d\leq a\leq \bar{C}I_d$.
Then there exists a constant $\nu_0$ such that, if $0\leq v\leq M$ in $B_R$ and
$$|\{v>(1-\alpha)M\}\cap B_R|<\nu_0|B_R|,$$
then $v\leq (1-\alpha/2)M$ in $B_{R/2}$, provided $M\geq R^{\frac{\ve}{2+\ve}}$. The constant $\nu_0$ only depends on the dimension, on $\alpha$, on $\ve$ and on the bounds and on the ellipticity of $a$ (i.e. on $\bar{c}$ and $ \bar{C}$). \end{lemma}

\begin{proof}
Set
$$k_n=M\left(1-\frac \alpha 2 - \frac{\alpha}{2^n}\right),\quad R_n=\frac R 2 + \frac{R}{2^n},\quad A_n=\left\{v>k_n\right\}\cap B_{R_n},$$
let $\eta_n$ be a cutoff function which vanishes outside $B_{R_n}$ and equals $1$ on $B_{R_{n+1}}$, with gradient bounds $|\nabla \eta_n|\leq CR^{-1}2^n$.
The function $\phi_n:=(v-k_n)_+$ is a subsolution of 
$$\nabla\cdot (a(x) \cdot \nabla \phi_n)=f'I_{v>k_n},$$ 
and we can test the equation against $\phi_n\eta_n^2$.
We get
$$\bar{c}\int |\nabla\phi_n|^2\eta_n^2 \leq \int <a\nabla\phi_n,\nabla\phi_n>\eta_n^2\leq  \bar{C}\int|\nabla\phi_n||\nabla\eta_n|\phi_n\eta_n +\int f' I_{v>k_n}\left(\phi_n\eta_n^2\right).$$
The first term in the right hand side may be estimated by $\frac{\bar{c}}{4} \int  |\nabla\phi_n|^2\eta_n^2+ \frac{ \bar{C}}{\bar{c}} \int \phi_n^2|\nabla\eta_n|^2$; in the second one can get rid of $I_{v>k_n}$ because anyway $\phi_n$ vanishes on $\{v\leq k_n\}$ and integrate by parts, so as to obtain
$$\int f \frac{\partial}{\partial x_1}\left(\phi_n\eta_n^2\right)\leq \int f|\nabla\phi_n|\eta_n^2+2\int f\phi_n\eta_n|\nabla\eta_n|\leq\frac{\bar{c}}{4}  \int  |\nabla\phi_n|^2\eta_n^2+ C\int f^2\eta_n^2+C\int \phi_n^2|\nabla\eta_n|^2.$$

Putting together the estimates one gets
$$\int |\nabla\phi_n|^2\eta_n^2 \leq C\left(\int \phi_n^2|\nabla\eta_n|^2+\int f^2\eta_n^2\right).$$
We then apply the estimate of Lemma \ref{poincare type} to the function $u=\phi_n\eta_n$ and get
$$\int \phi_n^2\eta_n^2 \leq C|A_n|\left(\int \phi_n^2|\nabla\eta_n|^2+\int |\nabla\phi_n|^2\eta_n^2\right)\leq  C|A_n|\left(\int \phi_n^2|\nabla\eta_n|^2+\int f^2\eta_n^2\right).$$
In the last expression, we may estimate $\phi$ with $M$ (actually $M(1-\alpha)$ would be sufficient), $|\nabla\eta_n|$ with $C2^n/R$, and apply H\"older inequality to $f$ and $I_{A_n}$. On the other hand, the values of $ \phi_n^2\eta_n^2$ may be estimated from below by $(k_{n+1}-k_n)I_{A_{n+1}}$. As a consequence, we get
$$(k_{n+1}-k_n)^2|A_{n+1}|\leq C\left[M^2|A_n|R^{-2}4^n+\left(\int f^{2+\ve}\right)^{\frac{2}{2+\ve}}|A_n|^{\frac{\ve}{2+\ve}}\right].$$
We define $Y_n$ as $|A_n|/R^2$ (i.e., up to possible bounded factors, the ratio between the area of $A_n$ and the ball itself). We have
$$\alpha^2M^24^{-n}Y_{n+1}\leq  CM^2Y_n\left[Y_n4^n+||f||_{L^{2+\ve}}^2Y_n^{\frac{\ve}{2+\ve}}M^{-2}R^{\frac{2\ve}{2+\ve}}\right].$$
Replacing $Y_n$ with $Y_n^{\frac{\ve}{2+\ve}}$ (since $Y_n\leq 4$), and supposing $M\geq R^{\frac{\ve}{2+\ve}}$, one gets an estimate of the form
$$Y_{n+1}\leq C(d, \bar{C},  \bar{c}, ||f||_{L^{2+\ve}},\alpha)Y_n^{1+\frac{\ve}{2+\ve}}16^n.$$
If one supposes $Y_1\leq \nu_0$, where $\nu_0$ is given by the assumptions of Lemma \ref{algebra}, we get $Y_n\to 0$. This means that the measure $|\{v>(1-\alpha/2)M\}\cap B_{R/2}|$ is zero, which is the thesis.
\end{proof}

\begin{lemma}\label{cerchio con valori alti} 
Let $0\leq u\leq M$ be a function in $H^1(B_R)$, assume that $|\{u > 3/4M \}\cap B_R|\geq\nu_0|B_R|$, and set $\ve_0=\sqrt{\nu_0/2}$, then one of the following alternatives happens:
\begin{itemize}
\item either the Dirichlet energy of $u$ in $B_R\setminus B_{\ve_0 R}$ is large: $\int_{B_R\setminus B_{\ve_0 R}} |\nabla u|^2\geq c M^2$,
\item or there exists a radius $s\in [\ve_0 R,R]$ such that $u >\frac 58 M$ on $\partial B_{s}$.
\end{itemize}
\end{lemma}
\begin{proof}

Thanks to the value we have chosen for $\ve_0$, the set $\{u>\frac 34M\}\cap (B_R\setminus B_{\ve_0R})$ must include a large part of $B_R\setminus B_{\ve_0 R}$ and in particular we have$| \{u>3/4M\}\cap (B_R\setminus B_{\ve_0R})|\geq\frac{\nu_0}{2}|B_R|$.

Let us define a subset $X$ of $[\ve_0R,R]$ through $X=\{s\in [\ve_0R,R]\;:\;\exists x\in\partial B_s\;:\;u(x)\geq \frac 34 M\}$. We obviously have 
$$2\pi R|X|\geq\int_X 2\pi s ds \geq |\{u> \frac 34M\}\cap (B_R\setminus B_{\ve_0R})|\geq\frac{\nu_0}{2}|B_R|$$
which gives $|X|\geq \frac{\nu_0}{4}R$.
Now, either the second alternative is verified or for every $s\in X$ there exists a point $x$ on the circle $\partial B_s$ such that $u(x)<\frac 58 M$. On each one of these circles we would get an oscillation of at least $M/8$ and hence
$$\frac M8\leq \int_{\partial B_s}|\nabla u|d\huno\leq (2\pi s)^{1/2}\left(\int_{\partial B_s}|\nabla u|^2d\huno\right)^{1/2},$$
which implies 
$$\int_{\partial B_s}|\nabla u|^2d\huno\geq \frac{M^2}{128\pi R}$$
and hence
$$\int_{B(x_0, R)\setminus B(x_0,\ve_0 R)} |\nabla u|^2\geq \int_X ds \int_{\partial B_s}|\nabla u|^2d\huno\geq \frac{M^2}{128\pi R}|X|\geq C\nu_0M^2.\qedhere$$
\end{proof}

\noindent From now on, we will denote by $\ve_0$ the constant defined in Lemma \ref{cerchio con valori alti} obtained for $\alpha=1/4$ and $\bar{c}=c_\delta$ (i.e. $\ve_0=\sqrt{\nu_0/2}$ where $\nu_0$ is obtained from Lemma \ref{reduction from above}).

\begin{lemma}\label{da poinnonso}
Let $u$ be in $H^1(B_R\setminus B_{\varepsilon_0R})$ and assume that  both the measures $|\{u > aM \}\cap (B_R\setminus B_{\varepsilon_0R})|$ and $|\{u < bM \}\cap (B_R\setminus B_{\varepsilon_0R})|$ are larger than $\eta R^2$, then 
$$\int_{B_R\setminus B_{\ve_0 R}}|\nabla u|^2\geq c(\eta, \ve_0,a, b)M^2.$$
\end{lemma}

\begin{proof}
Consider the function $\phi=\left(u- bM\right)_+\wedge(a-b)M$. It vanishes on a set whose measure is at least $\eta R^2$ and is maximal and equal to $(a-b)M$ on a set whose measure is, again, at least $\eta R^2$. This implies, thanks to Lemma \ref{poincare non so} applied to $\Omega=B_1\setminus B_{\ve_0}$, the inequality
$$\eta R^2(b-a)^2M^2\leq\int_{B_R\setminus B_{\ve_0 R}} \phi^2\leq C(\eta,\ve_0)R^2\int_{B_R\setminus B_{\ve_0 R}}|\nabla\phi|^2\leq C(\eta,\ve_0)R^2 \int_{B_R\setminus B_{\ve_0 R}}|\nabla u|^2.$$
The last inequality comes from the fact that $\phi$ is a $1-$Lipschitz function of $u$ and the statement is obtained.
\end{proof}

\begin{lemma}
Let $B_R$ and $B_{\ve_0R}$ be two concentric balls in $\Omega$ and $v_\delta=h_{1+\delta}(u_{x_1})$ as defined in Section 2. Then one of the three followings alternatives happens
\begin{itemize}
\item either $\osc(v_\delta,B_{\ve_0 R})\leq \frac 78 \osc(v_\delta,B_{R})$,
\item or the energy of $v_\delta$ in $B_{R}\setminus B_{\ve_0 R}$ is large: 
$$\int_{B_R\setminus B_{\ve_0 R}} |\nabla v_\delta|^2\geq c \left[\osc(v_\delta,B_R)\right]^2,$$
\item or we have $\osc(v_\delta,B_R)\leq R^{\frac{\ve}{2+\ve}}$.
\end{itemize}
\end{lemma}

\begin{proof}
Up to a translation, suppose that $0\leq v_\delta\leq M$ in $B_R$, where $M=\osc(v_\delta,B_R)$. In the ball $B_R$ the function $v_\delta$ (actually it is $v_\delta - \inf_{B_R}v_\delta$) is a subsolution of 
$$\nabla\cdot (a(x) \cdot \nabla v)=f'I_{v>0},$$
where the matrix $a$ satisfies $a\geq c_\delta I$ at least whenever $v>0$. Hence, one can replace $a$ with $c_\delta I$ on the set $\{v_\delta=0\}$ and suppose that the equation is uniformly elliptic. We start by supposing that the inequality $M>R^{\frac{\ve}{2+\ve}}$ is satisfied, otherwise the second alternative is realized. Assuming this estimate, one can apply Lemma \ref{reduction from above} with $\alpha=1/4$. We get that either $|\{v_\delta>3/4M\}\cap B_R|\geq\nu_0|B_R|$ or the oscillation on $B_{R/2}$ is reduced by a factor of $7/8$ and the first alternative is realized.

Suppose now that $|\{v_\delta>3/4M\}\cap B_R|\geq\nu_0|B_R|$. Lemma \ref{cerchio con valori alti}  proves that in this case either the third alternative is realized or there exists at least a radius such that $v_\delta$ stays high on a circle between $B_R$ and $B_{\ve_0 R}$.

Hence, suppose now that at least a value $s\in [\ve_0R,R]$ is such that $v_\delta\geq \frac 58 M$ on $\partial B_s$. In such a case, consider the function $w$ given by
$$w(x)=\begin{cases}(v_\delta-\frac 58 M)_-\wedge \frac {3}{16}M&\mbox{ if }x\in B_s,\\
				0&\mbox{ if }x\notin B_s.\end{cases}$$
To get information on the behavior of $w$ one needs to look back at the equation satisfied by the partial derivative $u_x$, instead of the one which has $v_\delta=(u_x-(1+\delta))_+$ as a subsolution. We know that $u_x$ satisfies
$$\nabla\cdot (a(x) \cdot \nabla u_x)=f',$$
which is a non-uniformly elliptic equation, and we test it against the test function $w_n=(w-k_n)_+$, where the numbers $k_n$ are given by $k_n=(1-2^{-(n+1)}) \frac{M}{8}$, for $n\geq 4$. 
We get
$$\int <a\nabla u_x,\nabla w_n>=-\int f\frac{\partial}{\partial x_1} w_n$$
and we notice that, where $\nabla w_n\neq 0$, one has $\nabla w_n=-\nabla u_x$, and also $v_\delta>\frac{7}{16}M>0$, which implies $a\geq c_\delta I$. Hence we have
$$c_\delta\int |\nabla w_n|^2\leq \int <a\nabla w_n,\nabla w_n>\leq \int f|\nabla w_n|\leq \frac{c_\delta}{2}\int |\nabla w_n|^2+\frac{C}{2c_\delta}\int f^2I_{\nabla w_n\neq 0}.$$
This implies, by using also Lemma \ref{poincare type}
$$\int  w_n^2\leq C|A_n|\int |\nabla w_n|^2\leq C|A_n|||f||_{L^{2+\ve}}^2|A_n|^{\frac{\ve}{2+\ve}},$$
where $A_n=\left\{w>k_n\right\}$. With the same kind of estimates as in Lemma \ref{reduction from above}, one gets
$$|k_{n+1}-k_n|^2|A_{n+1}|\leq C|A_n|^{1+\frac{\ve}{2+\ve}} ||f||_{L^{2+\ve}}^2$$
and, setting again $Y_n=A_n R^{-2}$ and using $k_{n+1}-k_n=M2^{-(n+1)}$, this gives
$$Y_{n+1}\leq 4^nC(|f||_{L^{2+\ve}},c_\delta,\bar{C})\frac{R^{\frac{2\ve}{2+\ve}}}{M^2}Y_n^{1+\frac{\ve}{2+\ve}}\leq 4^nC(|f||_{L^{2+\ve}},c_\delta,\bar{C})Y_n^{1+\frac{\ve}{2+\ve}}.$$
The last inequality comes from the assumption $M>R^{\frac{\ve}{2+\ve}}$.
Again, if one supposes $Y_0\leq \nu_1$ (where $\nu_1$ is given by Lemma \ref{algebra}, as in Lemma \ref{reduction from above}), one gets $Y_n\to 0$, which means that $w\leq \frac M8$ on $B_s$, provided $|\{w>\frac{M}{16}\}|<\nu_1|B_R|$. 

This means that there are two remainding cases: either the condition on the measure is fullfilled, or not. In the first case the oscillation has gone down to $\frac 78 M$ on $B_s$ and hence we can say that the oscillation on $B_{\ve_0 R}$ is less than $7/8$ of that on $B_R$ and the first alternative is realized.  

In the second case, we may estimate again the energy thanks to Lemma \ref{da poinnonso}. In this last case it is the third alternative which is realized.
\end{proof}

The conclusion of all these lemmas is a continuity result on $v_\delta$.

\begin{theorem}\label{princ}
Let $u$ be a solution of $\nabla\cdot (F(\nabla u))=f$, with $f\in L^{2+\ve}$ in dimension $d=2$, and suppose that $v_\delta=h_{1+\delta}(u_{x_1})_+\in H^1\cap L^\infty$ and that $D^2H(\nabla u)\leq \bar{C}I_d$ (which is always true if $H\in C^2$ and $\nabla u\in L^\infty$). Then $v_\delta$ is continuous on $\Omega$ and its modulus of continuity gives $|v_\delta(x)-v_\delta(y)|\leq C\left(\ln |x-y|\right)^{-1/2}$ for any $x,y\in \Omega_0\subset\Omega$, the constant $C$ depending on $||f||_{L^{2+\ve}}$, $||v_\delta||_{L^\infty},$ $||v_\delta||_{H^1}$, $c_\delta$, $\bar{C}$, $d(\Omega_0 ,\partial\Omega)$.
\end{theorem}

\begin{proof}
Take a sequence of balls $B_n$ included in $\Omega$, of the form $B_n=B(x_0,\ve_0^nR_0)$, ($R_0\leq d(\Omega_0,\partial\Omega)$ for $n=1,\dots, N$. We call $R$ the last radius, i.e. $R=R_0\ve_0^N$, and hence 
$$N=\frac{\ln(R/R_0)}{\ln\ve_0}.$$
Each time we pass from $B_n$ to $B_{n+1}$ either the oscillation is diminished by a factor $7/8$, or the oscillation was already smaller than a power $\beta=\ve/(2+\ve)$ of the radius, or the energy in the crown is controlled from below. Let us denote by $j,\,k$ and $h$ the number of times the three situations happen. We have $j+k+h=N$. If we call $M$ as before the oscillation of $v_\delta$ on $B_N$ we have the three following informations:
\begin{itemize}
\item $M(\frac{8}{7})^j\leq ||v_\delta||_{L^\infty}$, since the oscillation on $B_0$ cannot be more than the maximal value of $v_\delta$, which is positive, and it has diminished $j$ times; we can also write $M\frac j7\leq  ||v_\delta||_{L^\infty}$, since $(\frac{8}{7})^j=(1+\frac 17)^j\geq 1+\frac j 7\geq \frac j7$;
\item $M\leq (R_0\ve_0^i)^\beta\leq (R_0\ve_0^k)^\beta$, where $i$ is the last index of the second type (and hence one can say that $\ve_0^i\leq \ve_0^k$ since $i\geq k$); here we can write $M\leq C(R_0,\beta,\ve_0)/k$, thanks to the elementary inequality $\ve^{t}\leq C(\ve)/t$, which is valid for $\ve<1$ and $t\geq 0$.
\item $cM^2h\leq ||v_\delta||_{H^1}^2$, since on each crown $B_i\setminus B_{i+1}$ where the third alternative is realized we had a contribution proportional to the squared oscillation (and hence at least $M^2$) to the Dirichlet energy of $v_\delta$,
\end{itemize}
This implies the following bounds on $M$:
$$M\leq \min\left\{\frac{7||v_\delta||_{L^\infty}}{j}\,;\,\frac{||v_\delta||_{H^1}^2}{\sqrt{h}}\,;\,\frac{C(R_0,\beta,\ve_0)}{k}\right\}\leq \frac{\max\{7||v_\delta||_{L^\infty};\,||v_\delta||_{H^1}^2;\,C(R_0,\beta,\ve_0)\}}{\max\{\sqrt{j},\sqrt{h},\sqrt{k}\}}.$$
At least one of the three indices $j,k$ and $h$ is larger than $N/3$, and hence one gets
$$M\leq C(||v_\delta||_{L^\infty},||v_\delta||_{H^1}, R_0,\ve_0,\beta) \frac{1}{\sqrt{N}}=C(||v_\delta||_{L^\infty},||v_\delta||_{H^1}, R_0,\beta,\ve_0) \frac{\sqrt{|\ln\ve_0|}}{\sqrt{|\ln R|}},$$
which gives a logarithmic modulus of continuity for $v_\delta$.  This constant degenerates as the ellipticity constant $c_\delta$ goes to zero, since $\ve_0=\ve_0(c_\delta,\bar{C},\beta)$. The result is local and becomes worse when approaching the boundary, due to the dependence on $R_0$ (which can be taken as large as $d(\Omega_0,\partial\Omega)$).
\end{proof}

\section{From $F(\nabla u)$ to $v_\delta$}
In this section we show a useful and clarifying lemma, which aims at giving the relation between $F(\nabla u)$ and $v_\delta$.

\begin{lemma}\label{lipschitz}
For any $a\in\R^d$ consider all the points $z$ such that $F(z)=a$ and set
$$\gamma_\delta(a):= (z_1-(1+\delta))_+=h_{1+\delta}(z_1),$$ 
where $z_1$ denotes the first component of $z$. 

The function $\gamma_\delta$ is well-defined (i.e. if several $z$ satisfy $F(z)=a$ the value of $h_{1+\delta}(z_1)$ is the same for all of them), and it is Lipschitz continuous with a Lipschitz constant that does not exceed $c_\delta^{-1}$.
\end{lemma}
\begin{proof}
Consider points $a,b,z,w\in\R^d$ so that $F(z)=a$ and $F(w)=b$.
Suppose $z_1,w_1>1+\delta$: in this case one has, from $F=\nabla H$ and from the quantified convex behavior of $H$ on the half-space $\{z_1\geq 1+\delta\}$,
$$<F(z)-F(w),z-w>\geq c_\delta |z-w|^2,$$
which implies $|z-w|\leq c_\delta^{-1}|a-b|$ (this implies in particular uniqueness of $z$ and $w$). Then one can compose with $h_{1+\delta}$, which is $1-$Lipschitz, and get 
$$|h_{1+\delta}(z_1)-h_{1+\delta}(w_1)|\leq |z_1-w_1|\leq |z-w|\leq c_\delta^{-1}|a-b|.$$

In the case $z_1,w_1\leq 1+\delta$ one trivially has $h_{1+\delta}(z_1)=h_{1+\delta}(w_1)=0$. The only remaining case is (by simmetry) $z_1\leq1+\delta$ and $w_1>1+\delta$. In this case there is a unique point $\tilde{z}$ on the line through $z$ and $w$ lying on the hyperplane $\{z_1=1+\delta\}$.  Hence one may write
\begin{eqnarray*}
<F(\tilde{z})-F(w),\tilde{z}-w>&\geq& c_\delta |\tilde{z}-w|^2,\\
<F(z)-F(\tilde{z}),z-\tilde{z}>&\geq& 0,\\
<F(z)-F(\tilde{z}),\tilde{z}-w>&\geq& 0.
\end{eqnarray*}
The second inequality comes form simple convexity of $H$ and the last one from the second, since $z-\tilde{z}$ and $\tilde{z}-w$ share the same direction. Hence, summing up the first and the third, one gets
$$<F(z)-F(w),\tilde{z}-w>\geq c_\delta |\tilde{z}-w|^2,$$
which gives 
$$w_1-(1+\delta)\leq |\tilde{z}-w|\leq c_\delta^{-1}|a-b|,$$
and, again, the Lipschitz result.
\end{proof}

As we briefly addressed in Section 2, this lemma implies that $v_\delta$ is the composition of $F(\nabla u)$ with a Lipschitz function (the Lipschitz constant only depending on $\delta$). This is useful in several contexts. First of all, we need $v_\delta$ being a function (not necessarily Lipschitz) of  $F(\nabla u)$ if we tacitly approximate the equation, since we already underlined that the solutions of the approximated problem may converge to a specific solution $u$ of the limit one, but $F(\nabla u)$ is the same for all solutions. Thanks to this lemma, the same is true for $v_\delta$. Then, this lemma also proves that $v_\delta$ has the same Sobolev regularity of  $F(\nabla u)$, and that its $H^1$ norm only depends on that of $F(\nabla u)$ and on $c_\delta$. In particular, it does not depend on the direction that we choose for differentiating $u$ (that we conventionally called $x_1$ but that could have been any other direction). This will be useful in the next section, where we will deduce results on $|\nabla u|$ starting from $v_\delta$ (using the arbitrariness of the direction and the equicontinuity of all the functions $v_\delta$).

\section{Continuity of functions of $\nabla u$}

We have proven quantified continuity results for $v_\delta=h_{1+\delta}(\partial u /\partial x_1 )$. Just changing the variable one differentiate with respect to, one gets the same results for every function of the form $h_{1+\delta}(\nabla u\cdot e),$ for $|e|=1$. Since these functions all share the same modulus of continuity, taking the supremum over $e$ one has the same result for $h_{1+\delta}(|\nabla u|)=(|\nabla u|-(1+\delta))_+$.

Finally, if one lets $\delta\to 0$, one gets continuity for $(|\nabla u|-1)_+$ and for $(\nabla u\cdot e-1)_+$ as uniform limits of continuous functions. The modulus of continuity depends on the dependence of the modulus of continuity previously obtained with respect to $\delta$. We did not explicitly accounted for these continuity moduli but it could have been possible. Anyway, one can say that the continuity of $(|\nabla u|-1)_+$ and for $(\nabla u\cdot e-1)_+$ only depends on the constants $(c_\delta)_{\delta>0}$, on the upper bounds on the Hessian of $H$ (in the applications these usually depend on the maximum of $|\nabla u|$) and on the $L^{2+\ve}$ norm of $f$. All the results are local.

By composition, it is not difficult to prove continuity for functions of the form $g(|\nabla u|)$, for any real continuous function $g$ satisfying $g(t)=0$ for $t\leq 1$. 

To allow for more general functions and in particular for vector functions of $\nabla u$, we present the following lemma:
\begin{lemma}\label{gz}
Let $z:\Omega_0\to\R^d$ be a bounded function such that the family of functions  $h_1(z\cdot e)$ is equicontinuous for $e\in\R^d$ with $|e|=1$. Let $g:\R^d\to\R^d$ be a continuous function with $g=0$ on $B_1$. Then $g(z)$ is continuous on $\Omega_0$ and its modulus of continuity only depends on that of $g$, on $||z||_{L^\infty}$ and on the common modulus of continuity of the functions $h_1(z\cdot e)$.
\end{lemma}
\begin{proof}
First remark that, as we already pointed out, then the function $(|z|-1)_+$ as well is continuous, with the same modulus of continuity, as a supremum of a family of equicontinuous functions. 

Now, fix $\ve>0$ and let $\delta=\delta(\ve)>0$ be the corresponding $\delta$ given by the equicontinuity of all the functions $h_1(z\cdot e)$ and $h_1(|z|)$. Take two points $x,y\in\Omega_0$ with $|x-y|<\delta$ and set $a=z(x)$ and $b=z(y)$. 

We distinguish two cases: either $|a|\leq 1+2\ve$ or not. If yes, then we have $g(a)\leq \omega(2\ve)$, where $\omega:\R^+\to\R^+$ is the modulus of continuity of $g$, and
$$\left|(|a|-1)_+-(|b|-1)_+\right|<\ve,\quad \mbox{with }\;(|a|-1)_+\leq2\ve,$$
which implies $(|b|-1)_+\leq 3\ve$ and $|b|\leq 1+3\ve$ and $g(b)\leq \omega(3\ve)$. Finally, $|g(a)-g(b)|\leq \omega(2\ve)+\omega(3\ve)$.

If not, then we have $(|a|-1)_+=|a|-1>2\ve$ and $\left|(|a|-1)_+-(|b|-1)_+\right|<\ve$ implies $(|b|-1)_+>\ve$ and $(|b|-1)_+=|b|-1$. This gives $||a|-|b||<\ve$. In this case we want to estimate $|a-b|$: we use the continuity of $(z\cdot e-1)_+$ with $e=a/|a|$. We get
$$\left||a|-1-\left(b\cdot\frac{a}{|a|}-1\right)_+\right|<\ve$$
since, again, we have $|a|-1>2\ve$, then $(b\cdot\frac{a}{|a|}-1)_+=b\cdot\frac{a}{|a|}-1>\ve$ and hence
$$\left||a| - b\cdot\frac{a}{|a|}\right|<\ve,\quad\mbox{ and thus }\quad b\cdot\frac{a}{|a|}>|a|-\ve.$$
Then we write
\begin{eqnarray*}
|a-b|^2=|a|^2+|b|^2-2a\cdot b&<&|a|^2+|b|^2-2|a|(|a|-\ve)=2|a|\ve+|b|^2-|a|^2\\
&=&2|a|\ve+\left(|b|+|a|\right)\left(|b|-|a|\right)\leq\ve\left(2|a|+|a|+|b|\right)\leq \ve\left(4|a|+\ve\right).
\end{eqnarray*}
Since $z$ was supposed to be bounded, we can assume $|a|\leq M$ and hence $|a-b|^2\leq C\ve$. In this case we get $|g(a)-g(b)|\leq \omega(\sqrt{C\ve})$.

In all the cases we get continuity of $g\circ z$.
\end{proof}

The following theorem is just a corollary of Lemma \ref{gz}, Lemma \ref{lipschitz} and Theorem \ref{princ}.

\begin{theorem}
Let $u$ be a solution of $\nabla\cdot (F(\nabla u))=f$, in dimension $d=2$, where $F=\nabla H$ and $H$ is a convex function satisfying the assumption of Section 2, with $D^2H$ bounded on bounded sets. Suppose $f\in L^{2+\ve}$ and that $F(\nabla u)\in H^1\cap L^\infty$. Then $g(\nabla u)$ is continuous on $\Omega$ for any continuous function $g:\R^2\to\R$ such that $g=0$ on $B_1$.
\end{theorem}

\section{Applications}

As a particular consequence of the previous section, in the case $H=H_{(q)}$ (for $q\geq 2$), one has continuity for $F(\nabla u)$ itself, since the function $F$ vanishes on the ball $B_1$ and all the assumptions for applying the previous theorems are verified (in particular, \cite{BraCarSan} proves $H^1$ and $L^\infty$ regularity). 

As we said, $\sigma=F(\nabla u)$ optimizes $\int H^*(\sigma)$ under a divergence constraint, and in this case we have $ H_{(q)}^*(\sigma)= |\sigma|+\frac 1p |\sigma|^p$.  We summarize our results - under the variational formulation - in the following theorem.

\begin{theorem}
Fix $p\leq 2$.The unique minimizer $\bar{\sigma}$ of
\begin{equation}\label{min div fix}
\min \left\{ \int |\sigma|+\frac 1p |\sigma|^p\;:\; \sigma\in L^p(\Omega),\;\nabla\cdot\sigma = f,\;\sigma\cdot n =0\right\}
\end{equation}
is a continuous function in the interior of $\Omega$, provided $f\in W^{1,p}(\Omega).$
\end{theorem}
\begin{proof}
All the ingredients for this result have already been proved. We know that the optimal solution $\bar{\sigma}$ equals $F(\nabla u)$ where $F=\nabla  H_{(q)}$ (with $\frac 1q + \frac 1p = 1$ and $q\geq 2$)and $u$ is the solution of $\nabla\cdot (F(\nabla u))=f$. Thanks to what proven in \cite{BraCarSan}, the vector field $F(\nabla u)$ is both bounded and $H^1$. This last regularity result requires, for the proof which is performed in \cite{BraCarSan}, the Sobolev assumption on $f$. The $L^\infty$ result is presented in \cite{BraCarSan} for a Lipschitz $f$ but it should not be difficult to adapt to $f\in L^{2+\ve}$. On the contrary it does not seem easy to relax the assumption for Sobolev regularity.

Then we use Lemma \ref{lipschitz} to say that $v_\delta\in H^1(\Omega)$ as well, and that the $H^1$ norm of $v_\delta$ does not exceed $c_\delta^{-1}||F(\nabla)||_{H^1}$ (in particular, for fixed $\delta$, it is uniform with respect to the direction $e$). 

This allows to use the results of Sections 3 and 4 for getting first the continuity of $v_\delta$ and then of $F(\nabla u)$, since in this case $F$ vanishes on the ball $B_1$. 

Boundedness of $D^2H_{(q)}(\nabla u)$ comes from $H_{(q)}\in C^{1,1}_{loc}$ and boundedness of $\nabla u$.
\end{proof}

This result is interesting in itself as a regularity result for a variational problem under divergence constraints.

We also briefly present the consequences that our continuity proof has in the domain of continuous traffic congestion, as in \cite{CJS} and later in \cite{BraCarSan}.

The latter paper mainly aims at exploring the equivalence between Problem \eqref{min div fix} and the minimization of a total traffic intensity $\int H^*(i_Q(x))dx$ where $i_Q(x)$ denotes the quantity of traffic at $x$ associated to a probability  distribution $Q$ on the space of possible paths (the continuous framework is described through the choice of all Lipschitz curves in $\Omega$ as a set of admissible paths, instead of taking the paths on a given network). For modelling reasons, the choice of $H= H_{(q)}$ or in general of a function $H$ such that $(H^*)'(0)>0$ makes more sense than a simple power, since $(H^*)'(0)$ represents the cost for passing through a point with no traffic, which should be low but non-zero. 

It turns out from the analysis in \cite{BraCarSan} that one can build a measure $Q$ which is optimal for the traffic intensity problem starting from an optimal vector field $\bar{\sigma}$ for \eqref{min div fix}. We then follow the integral curves of a vector field 
$$\hat{\sigma}(t,x):=\frac{\bar{\sigma}(x)}{(1-t)f^+(x)+tf^-(x)},$$
where $f^+$ and $f^-$ are the positive and negative parts of $f$, respectively.  This would work fine provided the vector field is Lipschitz continuous, so that its integral curves are well-defined as the solutions of a standard Cauchy problem. Since $\bar{\sigma}$ may not always be assumed to be Lipschitz, due to the degeneracy of the equation that one gets with $H= H_{(q)}$, then the efforts in \cite{BraCarSan} were devoted to the proof of Sobolev regularity so as to apply DiPerna-Lions theory for weakly differentiable vector fields. 

Yet, in the case where $\bar{\sigma}$ turns out to be continuous (and the present paper proves such a result in dimension two), its integral curves may be at least interpreted as classical solution of an ODE with non-uniqueness and their regularity could be improved (solutions of a first-order Cauchy problem with continuous data are $C^1$).

Besides this first application, there is another interesting consequence in traffic congestion which is more linked to the topics of \cite{CJS}: it is proven in such a paper that optimal traffic distributions $Q$ satisfy a Wardrop equilibrium principium (which says that every path which is actually followed by somebody must be actually geodesics with respect to a metric taking into account traffic congestion: it is an equilibrium since the metric itself depends on $Q$ and $Q$ must be concentrated on $i_Q-$optimal curves). Yet, the metric one obtains is a function $g(i_Q(x))$ which is a priori neither continuous nor semi-continuous, since $i_Q$ is only $L^p$. The efforts of a large part of  \cite{CJS} aim at giving a meaning to geodesic distances and geodesic curves in such a setting. 

Yet, the optimal traffic measure $Q$ that one can build from $\sigma$ has the property that $i_Q=|\sigma|$ (see \cite{BraCarSan}) and this allows, provided $\sigma\in C^0$, to set the geodesic problem in a usual continuous framework.

 \end{document}